\documentclass[twoside,12pt]{article}

\usepackage{amssymb, amsmath}
\usepackage{mathtools}
\usepackage[thmmarks,amsmath]{ntheorem}

\usepackage[OT1]{fontenc}

\usepackage[utf8]{inputenc}
\usepackage[T1]{fontenc}
\usepackage{dsfont}
\usepackage{latexsym} 

\usepackage{enumerate}	
\usepackage{footnpag} 
\usepackage{makeidx}  

\usepackage{geometry}	
\usepackage{parcolumns}

\usepackage{graphicx}
\usepackage{color} 

\usepackage[all]{xy}

\usepackage{url}
\usepackage{hyperref}

\usepackage[OT2,OT1]{fontenc}
\newcommand\cyr
{
\renewcommand\rmdefault{wncyr}
\renewcommand\sfdefault{wncyss}
\renewcommand\encodingdefault{OT2}
\normalfont
\selectfont
}
\DeclareTextFontCommand{\textcyr}{\cyr}

\renewcommand{\geq}{\geqslant}	
\renewcommand{\leq}{\leqslant}
\def\mid|{\hs\middle|\hs}
\def\hs{\hspace*{0.1cm}}

\def\1{\mathbb{1}}
\newcommand{\sub}{\subseteq}	
\renewcommand{\1}{\mathds{1}}					

\newcommand{\I}{\mathcal{I}}

\newcommand{\NN}{\mathbb{N}}	
\newcommand{\RR}{\mathbb{R}} 
\newcommand{\ZZ}{\mathbb{Z}} 

\def\t{\textnormal}

\newcommand{\ohne}{\backslash}

\newcommand{\disp}{\displaystyle}		

\def\c{\cite }
\newcommand{\hide}[1]{}  

\def\phi{\varphi}
\def\rho{\varrho}

\DeclareMathAlphabet\mathbfcal{OMS}{cmsy}{b}{n}

\def\dd{^\textnormal{dd}}
\newcommand{\Inf}{\wedge}	%

\let\int\relax 
\DeclareMathOperator{\int}{int}

\renewcommand{\theequation}{\arabic{equation}}

\usepackage{enumitem}
\setenumerate[0]{label=(\roman*)}
\setenumerate[2]{label=(\alph*)}

\newcounter{Zaehler}
\theoremstyle{plain}
\theoremheaderfont{\normalfont\bfseries}
\theorembodyfont{\itshape}
\theoremseparator{.} 
\theorempreskip{\topsep}
\theorempostskip{\topsep} 
\theoremnumbering{arabic}
\theoremsymbol{}

\newtheorem{theorem}{Theorem}
\newtheorem{lemma}[theorem]{Lemma}
\newtheorem{corollary}[theorem]{Corollary}
\newtheorem{proposition}[theorem]{Proposition}
\theoremstyle{plain}
\theoremheaderfont{\normalfont\bfseries}
\theorembodyfont{\upshape}
\theoremseparator{.} 
\theorempreskip{\topsep}
\theorempostskip{\topsep} 
\theoremnumbering{arabic}
\theoremsymbol{}

\newtheorem{example}[theorem]{Example}
\newtheorem*{remark}{Remark}
\newtheorem{numremark}[theorem]{Remark}

\makeatletter
\newcommand{\eqnum}{\leavevmode\hfill\refstepcounter{equation}\textup{\tagform@{\theequation}}}
\makeatother

\theoremstyle{nonumberplain}
\theoremheaderfont{\itshape}
\theorembodyfont{\normalfont}
\theoremseparator{.} 
\theorempreskip{\topsep}
\theorempostskip{\topsep} 
\theoremsymbol{\ensuremath\square} 
\qedsymbol{$\Box$}

\newtheorem{proof}{Proof}

\begin{document}
\title{Order continuous operators on pre-Riesz spaces and embeddings}
\author{Anke Kalauch, Helena Malinowski}
\maketitle
\begin{abstract}
We investigate properties of order continuous operators on pre-Riesz spaces with respect to the embedding of the range space into a vector lattice cover or, in particular, into its Dedekind completion. We show that order continuity is preserved under this embedding for positive operators, but not in general. 

For the vector lattice $\ell_0^\infty$ of eventually constant sequences, we consider the pre-Riesz space of regular operators on $\ell_0^\infty$ and show that making the range space Dedekind complete does not provide a vector lattice cover of the pre-Riesz space. 
A similar counterexample is obtained for the directed part of the space of order continuous operators on $\ell_0^\infty$.
\end{abstract}




\section{Introduction}
For vector lattices $X$ and $Y$ with $Y$ Dedekind complete, the classical Riesz-Kantorovich theorem\footnote{See Theorem \ref{prelim.53} below.} states that the space $L_r(X,Y)$ of regular operators is a vector lattice. Moreover, Ogasawara showed that the space $L_{oc}(X,Y)$ of order continuous operators is a band in $L_r(X,Y)$ and, hence, a vector lattice\footnote{See Theorem \ref{prelim.1004} below.}. If $Y$ is an Archimedean vector lattice that lacks the Dedekind completeness, then $L_r(X,Y)$ and the directed part $L_{oc}^\diamond(X,Y)$ of $L_{oc}(X,Y)$ turn out to be Archimedean pre-Riesz spaces. Pre-Riesz spaces are exactly those partially ordered vector spaces that can be order densely embedded into vector lattices, their so-called vector lattice covers. For the pre-Riesz spaces of operators in question, one can ask up to which extent properties of $L_r(X,Y)$ and $L_{oc}^\diamond(X,Y)$ are reflected by the vector lattices $L_r(X,Y^\delta)$ and $L_{oc}(X,Y^\delta)$, respectively, where $Y^\delta$ denotes the Dedekind completion of $Y$. 

The paper is organized as follows. Section 2 contains all preliminary statements. 
In Section 3 we investigate order continuous operators in view of the extended range space. It turns out that the extension of the range space preserves the order continuity of an operator, whereas the similar statement for the restriction is only true for positive operators.  
In Section 4 we deal with the idea that at first glance it would make sense to establish the vector lattice $L_r(X,Y^\delta)$  as a vector lattice cover of $L_r(X,Y)$, and, similarly, $L_{oc}(X,Y^\delta)$ of $L_{oc}^\diamond(X,Y)$. We show that this approach fails, even if we take the vector lattice $X=Y=l_0^\infty$ of eventually constant sequences as the underlying space, which has nice additional properties as, e.g., an algebraic basis.

Both parts of the present paper indicate that a straightforward embedding of the range space into its Dedekind completion might essentially enlarge the space of operators, such that properties are not transferred appropriately. Thus vector lattice theory for spaces of operators can not be applied directly to pre-Riesz spaces of operators, and a theory of the latter is expected to be more involved.

\section{Preliminaries}
Let $X$ be a real vector space and let $X_+$ be a \emph{cone} in $X$, that is,  $X_+$ is a wedge ($x,y\in X_+$ and $\lambda,\mu\geq 0$ imply $\lambda x + \mu y \in X_+$) and $X_+ \cap (-X_+) =\left\{0\right\}$. In $X$ a partial order is defined by $x\leq y$ whenever $y-x\in X_+$. 
The space $(X,X_+)$ (or, loosely $X$) is then called a (\emph{partially}) \emph{ordered vector space}.
For a linear subspace $D$ of $X$ we
consider in $D$ the order induced from $X$, i.e.\ we set $D_+:=D\cap X_+$.

An ordered vector space $X$ is called \emph{Archimedean} if for every $x,y\in X$ with $nx\leq y$ for every $n\in\NN$ one has $x\leq 0$. 
Clearly, every subspace of an Archimedean ordered vector space is Archimedean. A subspace $D\sub X$ is called \emph{directed} if for every $x,y\in D$ there is an element $z\in D$ such that $x,y\leq z$.
An ordered vector space $X$ is directed if and only if  $X_+$ is \emph{generating} in $X$, that is, $X = X_+ - X_+$. A linear subspace $D$ of $X$ is \emph{majorizing} in $X$ if for every $x\in X$ there exists $d\in D$ with $x\leq d$. An ordered vector space $X$ has the \emph{Riesz decomposition property} (\emph{RDP}) if for every $x_1,x_2,z \in X_+$ with $z \leq x_1+x_2$ there exist $z_1,z_2 \in X_+$ such that $z = z_1+z_2$ with $z_1 \leq x_1$ and $z_2\leq x_2$. 
For a subset $M\sub X$ denote the set of all upper bounds of $M$ by $M^u=\left\{x\in X\mid| \forall m\in M\colon m\leq x\right\}$.

For standard notations in the case that $X$ is a vector lattice see \c{PosOp}. Recall that a vector lattice $X$ is \emph{Dedekind complete} if every non-empty subset of $X$ that is bounded above has a supremum.

Next we define disjointness, bands and ideals in ordered vector spaces.
The elements $x,y\in X$ are called \emph{disjoint}, in symbols $x\perp y$, if $\left\{x+y,-x-y\right\}^u = \left\{x-y,-x+y\right\}^u$, for motivation and details see \c{1}. If $X$ is a vector lattice, then this notion of disjointness coincides with the usual one, see \c[Theorem~1.4(4)]{PosOp}. Let $Y$ be an ordered vector space, $X$ an order dense subspace of $Y$, and $x,y\in X$. Then the disjointness notions in $X$ and $Y$ coincide, i.e.\ $x\perp y$ in $X$ holds if and only if $x\perp y$ in $Y$, see \c[Proposition~2.1(ii)]{1}.
The \emph{disjoint complement} of a subset $M\sub X$ is $M^{\t{d}} := \left\{x\in X \mid| \forall y\in M\colon x\perp y\right\}$. A linear subspace $B$ of an ordered vector space $X$ is called a \emph{band} in $X$ if $B\dd = B$, see \c[Definition~5.4]{1}. If $X$ is an Archimedean vector lattice, then this notion of a band coincides with the classical notion of a band in vector lattices (where a band is defined to be an order closed ideal). For every subset $M\sub X$ the disjoint complement $M^{\t{d}}$ is a band, see \c[Proposition~5.5]{1}. In particular, we have $M^{\t{d}} = M^{\t{ddd}}$.

The following notion of an ideal is introduced in \c[Definition~3.1]{vanGaa}. A subset $M$ of an ordered vector space $X$ is called \emph{solid} if for every $x\in X$ and $y\in M$ the relation
$\left\{x,-x\right\}^u \supseteq \left\{y,-y\right\}^u$ implies $x\in M$.
A solid subspace of $X$ is called an \emph{ideal}. This notion of an ideal coincides with the classical definition, provided $X$ is a vector lattice. For a subset $M\sub X$
the set $\mathcal{I}_M:= \bigcap \left\{I\sub X\mid| I \t{ is an ideal in } X \t{ with } M\sub I\right\}$
is an ideal, called the \emph{ideal generated by} $M$.

We call a linear subspace $D$ of an ordered vector space $X$ \emph{order dense} in $X$ if for every $x\in X$ we have \[x = \inf\left\{z\in D \mid| x\leq z\right\},\] that is, the greatest lower bound of the set $\left\{z\in D \mid| x\leq z\right\}$ exists in $X$ and equals $x$, see \c[p.~360]{113}. Clearly, if $D$ is order dense in $X$, then $D$ is majorizing in $X$.
Recall that a linear map $i\colon X\rightarrow Y$, where $X$ and $Y$ are ordered vector spaces, is called \emph{bipositive} if for every $x\in X$ one has $i(x) \geq 0$ if and only if $x\geq 0$. Every bipositive linear map is injective. 
In \c{vanHaa} the theory of pre-Riesz spaces is developed by van Haandel.
We call an ordered vector space $X$ a \emph{pre-Riesz space} if there exist a vector lattice $Y$ and a bipositive linear map $i\colon X\rightarrow Y$ such that $i(X)$ is order dense in $Y$. In this case, $(Y,i)$ is called a \emph{vector lattice cover} of $X$. 
If $X$ is a subspace of $Y$ and $i$ is the inclusion map, we write briefly $Y$ for $(Y,i)$.
By \c[Theorem 17.1]{vanHaa} every Archimedean directed ordered vector space is a pre-Riesz space. Moreover, every pre-Riesz space is directed.
If $X$ is an Archimedean directed ordered vector space, then every vector lattice cover of $X$ is Archimedean.
For the following result see \c[Corollary~5]{02}.
\begin{lemma}\label{properties.23y}
	Let $X$ be a pre-Riesz space, $(Y,i)$ a vector lattice cover of $X$ and $S\sub X$ a non-empty subset.
	\begin{enumerate}
		\item\label{properties.23y.eq1} If $\sup S$ exists in $X$, then $\sup i(S)$ exists in $Y$ and we have $\sup i(S) = i(\sup S)$.
		\item\label{properties.23y.eq2} If $\sup i(S)$ exists in $Y$ and $\sup i(S)\in i(X)$, then $\sup S$ exists in $X$ and we have $\sup i(S) = i(\sup S)$.
	\end{enumerate}
\end{lemma}
Clearly, Lemma~\ref{properties.23y} can be analogously formulated for infima instead of suprema.

By \c[Chapter~X.3]{VulikhWeber} every Archimedean directed  ordered vector space $X$ has a unique  Dedekind completion, which we denote by $X^\delta$. Clearly, $X^\delta$ is a vector lattice cover of $X$.
A combination of \cite[Theorem~4.14]{vanHaa} and \cite[Theorem~4.5]{vanHaa} yields the following result.
\begin{theorem}\label{properties.20}
	Let $Y$ be an Archimedean directed ordered vector space and $X$ a directed order dense subspace of $Y$. Then the Dedekind completions of $X$ and $Y$ are order isomorphic vector spaces, i.e.\ we can identify $X^\delta = Y^\delta$.
\end{theorem}
In particular, for an Archimedean pre-Riesz space $X$ and its (Archimedean) vector lattice cover $Y$ we have $X^\delta = Y^\delta$.

Let $X$ be a pre-Riesz space and $(Y,i)$ a vector lattice cover of $X$.
The space $X$ is called \emph{pervasive in} $Y$, if for every $y\in Y_+$, $y\neq 0$, there exists $x\in X$ such that $0<i(x) \leq y$. By \c[Proposition~3.3.20]{Kalauch} the space $X$ is pervasive in $Y$ if and only if $X$ is pervasive in any vector lattice cover. So, $X$ is simply called \emph{pervasive}.

Next we discuss the restriction property and the extension property for ideals and bands.
Let $X$ be a pre-Riesz space and $(Y,i)$ a vector lattice cover of $X$. 
For $S\sub Y$ we write  $[S]i:=\left\{x\in X\mid| i(x)\in S\right\}$. 
The pair $(L,M)\subseteq \mathcal{P}(X)\times \mathcal{P}(Y)$
is said to satisfy  
\begin{itemize}
	\item[-]
	the \emph{restriction property} (R), if whenever $J\in M$,
	then $[J]i\in L$, and
	\item[-]
	the \emph{extension property} (E), if whenever $I\in L$,
	then there is $J\in M$ such that 
	$I=[J]i$.
\end{itemize}
In \c{1} the properties (R) and (E) are investigated for ideals and bands.
It is shown 
that the extension property (E) is satisfied for  bands, i.e.\ for $L$ being the set of bands in $X$ and $M$ being the set of bands in $Y$.
Moreover, the restriction property (R) is satisfied for ideals. In general, bands do not have (R) and ideals do not have (E).  
The appropriate sets $M$ and $L$ of directed ideals satisfy (E). If $X$ is pervasive, then by \c[Proposition~2.5 and Theorem~2.6]{3} we have (R) for bands.

If for an ideal $I$ in $X$ and an ideal $J$ in $Y$ we have $I=[J]i$, then $J$ is called an \emph{extension ideal} of $I$. An \emph{extension band} $J$ for a band $I$ in $X$ is defined similarly. Extension ideals and bands are not unique, in general.
If an ideal $I$ in $X$ has an extension ideal in $Y$, then
$\hat{I}:=\bigcap\left\{J\subseteq Y \mid| J \t{ is an extension ideal of } I\right\}$
is the smallest extension ideal of $I$ in $Y$.
The following result can be found in \c[Theorem~18]{03}.
\begin{theorem}\label{something}
	Let $X$ be a pre-Riesz space with a vector lattice cover $(Y,i)$ and let $I\sub X$ be an ideal. Then the following statements are equivalent.
	\begin{enumerate}
		\item\label{something.i1} $I$ is directed.
		\item\label{something.i3} There exists a set $S\sub X_+$ of positive elements such that $I=\I_S$.
		\item\label{something.i2} The set $i(I)$ is majorizing in $\I_{i(I)}$.\index{majorizing}
		\item\label{something.i4} $I$ has an extension ideal and for the smallest extension ideal $\hat{I}$ of $I$ there exists a set $S\sub X_+$ such that $\hat{I}=\I_{i(S)}$.
	\end{enumerate}
	If one of the previous statements is true, then in \ref{something.i3} and \ref{something.i4} we can choose $S:=I_+$.
\end{theorem}

The following result from \c[Theorem~29]{03} gives conditions under which a directed ideal is order dense in its smallest extension ideal.
\begin{theorem}\label{closedness.1001}
	Let $X$ be an Archimedean pervasive pre-Riesz space with a vector lattice cover $(Y,i)$. Let $I\sub X$ be a directed ideal and $\hat{I}\sub Y$ the smallest extension ideal of $I$. Then $(\hat{I},i|_{I})$ is a vector lattice cover of the pre-Riesz space $I$, and $I$ is pervasive.
\end{theorem}

The following well-known property of ideals can be found in \c[Theorem~IV.1.2]{Vulikh_en}.
\begin{lemma}\label{properties.0}
Let $X$ be a Dedekind complete vector lattice and $I\sub X$ an ideal. Then $I$ is a Dedekind complete sublattice of $X$.
\end{lemma}

In an ordered vector space, a net $(x_\alpha)_\alpha$ is said to be \textit{decreasing} (in symbols $x_\alpha\downarrow$), whenever $\alpha\leq\beta$ implies $x_\alpha \geq x_\beta$. For $x\in X$ the notation $x_\alpha \downarrow x$ means that $x_\alpha \downarrow$ and $\inf_\alpha x_\alpha = x$. The meaning of $x_\alpha \uparrow$ and $x_\alpha \uparrow x$ are defined analogously. We say that a net \textit{order converges}, or short \textit{o-converges}, to $x\in X$ (in symbols  $x_\alpha \stackrel{o}{\longrightarrow} x$), if there is a net $(z_\alpha)_\alpha$ in $X$ such that $z_\alpha\downarrow 0$ and for every $\alpha$ one has $\pm (x-x_\alpha) \leq z_\alpha$. The equivalence of $x_\alpha \stackrel{o}{\longrightarrow} x$ and $x-x_\alpha \stackrel{o}{\longrightarrow} 0$ is obvious. If a net o-converges, then its limit is unique. A set $M\sub X$ is called \textit{order closed}, or short \textit{o-closed}, if for each net $(x_\alpha)_\alpha$ in $M$ which o-converges to $x\in X$ one has $x\in M$.

The subsequent result is established in \c[Proposition~5.1]{2} and is, in particular, true for pre-Riesz spaces, as they can be identified with order dense subspaces of their vector lattice covers.
\begin{theorem}\label{prelim.15x}
Let $X$ be an order dense subspace of an ordered vector space $Y$. 
\begin{enumerate}
\item\label{prelim.15x.it1} If a net $(x_\alpha)_\alpha$ in $X$ and $x\in X$ are such that $x_\alpha\xrightarrow{o}x$, then $x_\alpha\xrightarrow{o}x$ in $Y$.
\item\label{prelim.15x.it2} If $J\sub Y$ is o-closed in $Y$, then $J\cap X$ is o-closed in $X$.
\end{enumerate}
\end{theorem}

For ordered vector spaces $(X,X_+)$ and $(Y,Y_+)$ we denote by $L(X,Y)$ the vector space of all linear operators from $X$ to $Y$. An operator $T\in L(X,Y)$ is called \emph{positive}, if $T(X_+)\sub Y_+$; \emph{regular}, if $T$ is a difference of two positive operators; \emph{order bounded} if $T$ maps order bounded sets in $X$ into order bounded sets in $Y$; \emph{order continuous}, if for every net $(x_\alpha)_\alpha$ in $X$ with $x_\alpha\xrightarrow{o} x \in X$ we have $T(x_\alpha)\xrightarrow{o} T(x)$. 
By $L(X,Y)_+$ we denote the set of all positive operators, by $L_r(X,Y)$ the set of all regular operators, by $L_b(X,Y)$ the set of all order bounded  operators and by $L_{oc}(X,Y)$ the set of all order continuous  operators in $L(X,Y)$.
If $X$ is directed, then $(L(X,Y), L(X,Y)_+)$ is an ordered vector space. We write $L_{oc}(X,Y)_+:=L(X,Y)_+\cap L_{oc}(X,Y)$.
The following result is well-known for vector lattices. Subsequently, we will use it without reference.
\begin{lemma}\label{operators.1005}
	Let $X$ and $Z$ be ordered vector spaces, $Z$ Archimedean, and $T\colon X\rightarrow Z$ a positive operator. Then the following statements are equivalent.
	\begin{enumerate}
		\item\label{operators.1005.it1} $T$ is order continuous.\index{operator!order continuous}\index{order continuous operator}
		\item\label{operators.1005.it2} For every net $(x_\alpha)_\alpha$ in $X$ with $x_\alpha\xrightarrow{o} 0$ it follows $T(x_\alpha)\xrightarrow{o} 0$.
		\item\label{operators.1005.it3} For every net $(x_\alpha)_\alpha$ in $X$ with $x_\alpha\downarrow 0$ it follows $T(x_\alpha)\downarrow 0$.
		\item\label{operators.1005.it4} For every net $(x_\alpha)_\alpha$ in $X$ with $x_\alpha\uparrow x$ it follows $T(x_\alpha)\uparrow T(x)$.
	\end{enumerate}
\end{lemma}
\begin{proof}
The equivalence of \ref{operators.1005.it1} and \ref{operators.1005.it2} is clear.
		
\ref{operators.1005.it2} $\Rightarrow$ \ref{operators.1005.it3}: Let $(x_\alpha)_\alpha$ in $X$ be a net with $x_\alpha\downarrow 0$. Since $T$ is a positive operator, $x_\alpha\downarrow$ implies $T(x_\alpha)\downarrow$.
Moreover, $x_\alpha\downarrow 0$ leads, in particular, to $x_\alpha\xrightarrow{o} 0$ and so $T(x_\alpha)\xrightarrow{o} 0$. That is, $\pm T(x_\alpha)\leq z_\alpha\downarrow 0$ holds for a net $(z_\alpha)_\alpha$ in $Z$. Since $T$ is positive, we have $0\leq T(x_\alpha) \leq z_\alpha\downarrow 0$. As $Z$ is Archimedean, it follows $\inf\left\{T(x_\alpha)\mid| \alpha\in  A\right\}=0$. We obtain $T(x_\alpha)\downarrow 0$.
	
\ref{operators.1005.it3} $\Rightarrow$ \ref{operators.1005.it2}: Let a net $(x_\alpha)_\alpha$ in $X$ be such that $x_\alpha\xrightarrow{o} 0$. Then there exists a net $(y_\alpha)_\alpha$ in $X$ with $\pm x_\alpha\leq y_\alpha \downarrow 0$. Due to $T$ being positive and by virtue of \ref{operators.1005.it3} we obtain $\pm T(x_\alpha)\leq T(y_\alpha)\downarrow 0$ in $Z$. It follows $T(x_\alpha)\xrightarrow{o} 0$. 
	
	\ref{operators.1005.it1} $\Rightarrow$ \ref{operators.1005.it4}: Let a net $(x_\alpha)_\alpha$ in $X$ be such that $x_\alpha\uparrow x$. Since $T$ is a positive operator, the relation $x_\alpha\uparrow$ implies $T(x_\alpha)\uparrow$.
	Moreover, $x_\alpha\uparrow x$ implies $x_\alpha\xrightarrow{o} x$ and so $T(x_\alpha)\xrightarrow{o} T(x)$. That is, $\pm (T(x)-T(x_\alpha))\leq z_\alpha\downarrow 0$ holds for a net $(z_\alpha)_\alpha$ in $Z$. Since $T$ is positive, we have $T(x_\alpha)\leq T(x)$ and thus $0\leq T(x)-T(x_\alpha) \leq z_\alpha\downarrow 0$. As $Z$ is Archimedean, it follows $\inf\left\{T(x)-T(x_\alpha)\mid| \alpha\in A\right\}=0$. We obtain $T(x)-T(x_\alpha)\downarrow 0$. Equivalently, $-T(x_\alpha)\downarrow -T(x)$ and so $T(x_\alpha)\uparrow T(x)$.
	
	\ref{operators.1005.it4} $\Rightarrow$ \ref{operators.1005.it3}: Let a net $(x_\alpha)_\alpha$ in $X$ be such that $x_\alpha\downarrow 0$. Then it follows $-x_\alpha \uparrow 0$ and so $-T(x_\alpha)=T(-x_\alpha)\uparrow 0$. Hence $T(x_\alpha)\downarrow 0$.
\end{proof}
Notice that for the equivalence of \ref{operators.1005.it1} and \ref{operators.1005.it2} the operator $T$ need not be positive.

The following theorem can be found e.g.\ in \c[Theorem~1.59]{AliTou}.
\begin{theorem}[Riesz-Kantorovich]\label{prelim.53}
If an ordered vector space $X$ with a generating cone has the RDP, then for every Dedekind complete vector lattice $Z$ the ordered vector space $L_b(X,Z)$ is also a Dedekind complete vector lattice -- and therefore, in this case, we have $L_r(X,Z)=L_b(X,Z)$.\\
\end{theorem}

The next result by Ogasawara can be found in its classical form in \c[Theorem~1.57]{PosOp}, where it is stated for operators on vector lattices with a Dedekind complete range space. The more general version below was established recently in  \c[Theorem~7.8]{10}.
\begin{theorem}[The generalized Ogasawara Theorem]\label{prelim.1004}
Let $X$ and $Z$ be pre-Riesz spaces, where $X$ has the RDP and $Z$ is Dedekind complete. Then $L_{oc}(X,Z)$ is a band in $L_b(X,Z)$.
\end{theorem}

\section{Order continuous operators: Extension and restriction of the range space}

A classical result \c[Lemma~1.54]{PosOp} states that every order continuous operator between vector lattices is order bounded.
The proof of this statement can be adopted for ordered vector spaces. 
\begin{proposition}\label{operators.3}
	Let $X$ and $Z$ be ordered vector spaces and let $X$ be directed. Then every order continuous operator $T\colon X\rightarrow Z$ is order bounded. That is, $L_{oc}(X,Z)\sub L_b(X,Z)$.
\end{proposition}
\begin{proof}
	Let $T\colon X\rightarrow Z$ be an order continuous operator. First we consider  for $x\in X_+$ the order interval $A:=[0,x]$.
	Let $A$ be endowed with the reversed order and let $(x_\alpha)_{\alpha\in A}$ be the net with $x_\alpha:=\alpha$ for every $\alpha\in A$. 
	Then $x_\alpha\downarrow 0$.  Due to the order continuity of $T$ there exists a net $(y_\alpha)_{\alpha\in A}$ of $Z$ such that $\pm T(x_\alpha)\leq y_\alpha\downarrow 0$. Consequently, for every $\alpha\in[0,x]$ we have $\pm T(\alpha) = \pm T(x_\alpha) \leq y_\alpha \leq y_x$, i.e.\ $T$ maps $[0,x]$ into the order interval $[-y_x,y_x]\sub Z$. For an arbitrary order interval $[a,b]\sub X$ we have $[a,b]=[0,b-a]+a$. As $T$ is linear, it maps $[a,b]$ 
	into an order bounded subset of $Z$.
\end{proof}

Let $X$ and $Z$ be ordered vector spaces, where $X$ is directed. We define \emph{the directed part} of $L_{oc}(X,Z)$ by
\begin{equation}
L_{oc}^\diamond(X,Z):=L_{oc}(X,Z)_+ - L_{oc}(X,Z)_+.
\end{equation}

First, we collect some properties of $L^\diamond_{oc}(X,Z)$. Clearly, we have the inclusion $L_{oc}^\diamond(X,Z) \sub L_{oc}(X,Z)\cap L_r(X,Z)$. 
In the classical situation, i.e.\ if $X$ is a directed ordered vector space with RDP and $Z$ a Dedekind complete vector lattice, by the Riesz-Kantorovich Theorem~\ref{prelim.53} the space $L_r(X,Z)=L_b(X,Z)$ is a Dedekind complete vector lattice. Moreover, due to Theorem~\ref{prelim.1004}, the space $L_{oc}(X,Z)$ is a band in $L_r(X,Z)$ and thus directed, which yields the next result.

\begin{proposition}\label{atomic.0+1}
Let $X$ be a directed ordered vector space with RDP and let $Z$ be a Dedekind complete vector lattice. Then we have
\[L_{oc}^\diamond(X,Z)=L_{oc}(X,Z)\cap L_r(X,Z)=L_{oc}(X,Z).\]
\end{proposition}

In Remark~\ref{some_remark}(b) below we will see that $L_{oc}^\diamond(X,Z)$ and $L_{oc}(X,Z)$ do not coincide, in general.

\begin{proposition}\label{atomic.0+2}
	Let $X$ and $Z$ be ordered vector spaces with $X$ directed and $Z$ Archimedean. Then $L_{oc}^\diamond(X,Z)$ is an Archimedean pre-Riesz space.
\end{proposition}
\begin{proof}
It is clear that $L_{oc}^\diamond(X,Z)$ is an ordered vector space with a generating cone, i.e.\ the space is directed. Moreover, if $Z$ is Archimedean, then $L(X,Z)$ is Archimedean. Thus $L_{oc}^\diamond(X,Z)\sub L(X,Z)$ is likewise Archimedean. Consequently, $L_{oc}^\diamond(X,Z)$ is pre-Riesz.
\end{proof}

\begin{remark}
As it is open whether $L_{oc}(X,Z)\cap L_r(X,Z)$ is directed, hereafter we only consider assumptions as in Proposition~\ref{atomic.0+2} and deal with the pre-Riesz space $L_{oc}^\diamond(X,Z)$. 
\end{remark}

Next we investigate order continuous operators in view of the extension of the range.
\begin{theorem}\label{operators.6}
Let $X$ be a directed ordered vector space and $Z$ a pre-Riesz space with a vector lattice cover $(Y,i)$. If $T\in L_{oc}(X,Z)$, then $i\circ T\in L_{oc}(X,Y)$.
\end{theorem}
\begin{proof}
As $X$ is directed, $L_{oc}(X,Z)$ and $L_{oc}(X,Y)$ are ordered vector spaces.

Let $T\in L_{oc}(X,Z)$ and let  $(x_\alpha)_\alpha$ be a net in $X$ with $x_\alpha\xrightarrow{o} 0$. Then $T(x_\alpha)\xrightarrow{0} 0$ in $Z$, i.e.\ there exists a net $(z_\alpha)_\alpha$ in $Z$ such that $\pm T(x_\alpha) \leq z_\alpha \downarrow 0$.
Considering the vector lattice cover $(Y,i)$ of $Z$, by Lemma~\ref{properties.23y}\ref{properties.23y.eq1} we obtain  $\inf\left\{i(z_\alpha)\mid|\alpha\in A\right\}=i(\inf\left\{z_\alpha\mid|\alpha\in A\right\}) = i(0) = 0$. This implies
$\pm i(T(x_\alpha))\leq i(z_\alpha) \downarrow 0$,
i.e.\ the operator $i\circ T$ is order continuous.
\end{proof}

One can formulate the result in Theorem~\ref{operators.6} in the setting of order dense subspaces.
\begin{corollary}\label{fizzehn}
Let $X$ be a directed ordered vector space, $Y$ a vector lattice and $Z$ an order dense subspace of $Y$. Then $L_{oc}(X,Z)\subseteq  L_{oc}(X,Y)$.
\end{corollary}	

In Example~\ref{operators.7a} below we will need the following statement, where we consider an Archimedean pre-Riesz space $Y$ as an order dense subspace of its Dedekind completion $Y^{\delta}$ and apply Corollary~\ref{fizzehn}. 
\begin{corollary}\label{operators.xxx}
Let $X$ be a directed ordered vector space and $Y$ an Archimedean pre-Riesz space. Then $L_{oc}(X,Y)\subseteq L_{oc}(X,Y^\delta)$.
	\end{corollary}

Next we establish that the converse of Theorem~\ref{operators.6} is satisfied, provided the considered operator is positive.
\begin{theorem}\label{operators.7}
Let $X$ and $Z$ be pre-Riesz spaces and $(Y,i)$ a vector lattice cover of $Z$. For a positive operator $T\colon X\rightarrow Z$ let $i\circ T\in L_{oc}(X,Y)$. Then $T\in L_{oc}(X,Z)$. 
\end{theorem}
\begin{proof}
As $T$ is positive, it suffices to show 
that $x_\alpha\downarrow 0$ implies 
$T(x_\alpha)\downarrow 0$. 
 Let $(x_\alpha)_{\alpha \in A}$ be a net in $X$ with $x_\alpha\downarrow 0$, then
 $i(T(x_\alpha))\downarrow 0$. By Lemma~\ref{properties.23y}\ref{properties.23y.eq2} we obtain that the infimum of the set $\left\{T(x_\alpha)\mid|\alpha\in A\right\}$ exists in $Z$ with $\inf\left\{T(x_\alpha)\mid|\alpha\in A\right\}=0$. 
 As $i$ is bipositive, we get $T(x_\alpha)\downarrow 0$. Consequently,   $T$ is order continuous.
\end{proof}

\begin{corollary}\label{ocdiamondisnice}
	Let $X$ be a pre-Riesz space, $Y$ a vector lattice and $Z$ an order dense subspace of $Y$.
	Then we have
	$L(X,Z)\cap L_{oc}^\diamond(X,Y)=L_{oc}^\diamond(X,Z)$.
\end{corollary}

Theorem~\ref{operators.6} yields the inclusion $L(X,Z)\cap L_{oc}(X,Y)\supseteq L_{oc}(X,Z)$. In the following Example~\ref{operators.7a} we show that the converse of Theorem~\ref{operators.6} is not true, in general.
Moreover, in the example we have
\[L(X,Z)\cap L_{oc}(X,Y)\not\sub L_{oc}(X,Z),\]
i.e.\ the statement in Corollary \ref{ocdiamondisnice} is not true for sets of order continuous operators, respectively.
\begin{example}\label{operators.7a}
\textit{
	For an Archimedean vector lattice $X$ and an Archimedean pre-Riesz space $Z$, we give an example of an operator $T\in L(X,Z)\cap L_{oc}(X,Z^\delta)$ with $T\notin L_{oc}(X,Z)$.}

Consider the Archimedean vector lattice
\[Y:=\left\{(y_i)_{i\in\ZZ}\in\ell^\infty(\ZZ) \mid| \lim_{i\to\infty} y_i \t{ exists}\right\}\]
with point-wise order and its subspace
\[Z:=\left\{(z_i)_{i\in\ZZ}\in Y \mid| \sum_{k=1}^\infty \frac{z_{-k}}{2^k} = \lim_{i\to\infty} z_i\right\}.\]
In \c[Example~5.2]{2} it is shown that $Z$ is order dense in $Y$. 
Moreover, a sequence in $Z$ is given which does not order converge in $Z$ but order converges in $Y$. For an appropriate vector lattice $X$, we use a similar idea to provide an operator $T\colon X\rightarrow Z$ which is not order continuous, but is in $L_{oc}(X,Z^\delta)$.
Instead of showing $T\in L_{oc}(X,Z^\delta)$ we establish that $T\in L_{oc}(X,Y)$. Then by Theorem~\ref{properties.20} we can identify $Y^\delta$ and $Z^\delta$. 
 Due to Corollary~\ref{operators.xxx} 
 we then obtain $T\in L_{oc}(X,Z^\delta)$.

(a) Let
\[X=\ell_0^\infty:=\left\{(x_i)_{i\in\NN}\mid| \exists c\in\RR\hs \exists k\in\NN\hs \forall i\geq k\colon x_i=c \right\}\]
be the vector lattice of all real eventually constant sequences endowed with point-wise order. We define 
\[T\colon X\rightarrow Z,\quad (x_n)_{n\in\NN}\mapsto (x_n-x_{n-1})_{n\in\ZZ}\]
with the convention $x_k=0$ for every $k\in\ZZ\ohne\NN$.
For every $x\in X$ it follows that the sequence $T(x)$ is eventually zero. Moreover, $T(x)_k=0$ holds for every $k\in\ZZ\ohne\NN$. Therefore $T(x)\in Z$. 

We will first show that $T\in L_{oc}(X,Y)$.
Observe that $T$ can be decomposed into two 
positive operators, i.e.\ $T=T_1-T_2$, where for $x=(x_n)_{n\in\NN} \in\ell_0^\infty$ we define
\begin{align*}
T_1\colon &X \rightarrow Y,   &T_2\colon &X \rightarrow Y,\\
&x \mapsto (x_n)_{n\in\ZZ} & &x \mapsto (x_{n-1})_{n\in\ZZ}
\end{align*}
with the convention $x_n:=0$ for $n\in\ZZ\ohne\NN$. For every $x\in X$ the sequences $T_1(x)$ and $T_2(x)$ are convergent for $n\rightarrow\infty$ and therefore are contained in $Y$. 
We establish the order continuity of $T_1$ and $T_2$ in (c) below to obtain $T\in L_{oc}(X,Y)$. For this, we next provide in (b) properties of the order convergence in $X$ and in $Y$.

(b) (i) A net $(x^{(\alpha)})_\alpha$ in $X$ satisfies $x^{(\alpha)} \downarrow 0$ if and only if $x^{(\alpha)}_n \downarrow_\alpha 0$ for every $n\in\NN$. Indeed, let  $(x^{(\alpha)})_\alpha\in X$ with $x^{(\alpha)} \downarrow 0$. Then $x^{(\alpha)}_n \downarrow_\alpha$ for every $n\in\NN$. Assume that there is $N\in\NN$ such that $x^{(\alpha)}_N \downarrow_\alpha c>0$. Let $s\in X$ be the sequence which is zero in all components except in the component $N$ where it is $c>0$. Then it follows $s\leq x^{(\alpha)}$ for every $\alpha$, a contradiction to $x^{(\alpha)} \downarrow 0$. The converse implication is immediate.

(ii) If for a net $(x^{(\alpha)})_\alpha$ in $Y$ we have for every component $n\in\ZZ$ that $x^{(\alpha)}_n\downarrow_\alpha 0$, then it follows $x^{(\alpha)}\downarrow 0$. Indeed, from $x^{(\alpha)}_n\downarrow_\alpha$ for every $n\in\ZZ$ we obtain $x^{(\alpha)}\downarrow$. Assume that there is an $x=(x_k)_{k\in\ZZ}\in Y$ with $0<x\leq x^{(\alpha)}$ for every $\alpha$. Then there exists $k\in\ZZ$ such that $x_k\neq 0$, a contradiction to $x^{(\alpha)}_k\downarrow_\alpha 0$.

(iii) Consider the sequence $x=(x^{(n)})_{n\in\NN}$ of elements $x^{(n)}=(x^{(n)}_k)_{k\in\ZZ}\in Y$ such that for every fixed $n\in\NN$ we have $x^{(n)}_n=1$, $x^{(n)}_{n+1}=-1$ and $x^{(n)}_k=0$ otherwise. Then $x$ order converges in $Y$ to zero, i.e.\ $x^{(n)}\xrightarrow{o} 0$. Indeed, let $y=(y^{(n)})_{n\in\NN}$ be a sequence in $Y$ such that for every fixed $n\in\NN$
we have $y^{(n)}_k=1$ for every $k\geq n$ and $y^{(n)}_k=0$ otherwise. 
By (ii) we obtain $y^{(n)} \downarrow 0$. It follows $\pm x^{(n)} \leq y^{(n)} \downarrow 0$, i.e.\ we have $x^{(n)}\xrightarrow{o} 0$ in $Y$.

(c) We show the order continuity of $T_1$. As $T_1$ is positive, it suffices to show that for every net $(x^{(\alpha)})_\alpha$ in $X$ with $x^{(\alpha)} \downarrow 0$ it follows $T_1(x^{(\alpha)})\downarrow 0$ in $Y$. Let $x=(x^{(\alpha)})_\alpha$ be a net in $X$ with $x^{(\alpha)} \downarrow 0$. For the component $T_1(x^{(\alpha)})_n$ of the net $T_1(x^{(\alpha)})$ we have $T_1(x^{(\alpha)})_n = 0$ if $n\in\ZZ\ohne\NN$ and $T_1(x^{(\alpha)})_n = x^{(\alpha)}_n$ if $n\in\NN$. By (b)(i) it follows for each component $n\in\NN$ that $T_1(x^{(\alpha)})_n=x^{(\alpha)}_n \downarrow_\alpha 0$. Applying (b)(ii) we obtain in $Y$ that $T_1(x^{(\alpha)})\downarrow 0$, i.e.\ $T_1$ is order continuous.

Analogously we find that if a net $(x^{(\alpha)})_\alpha$ in $X$ is such that $x^{(\alpha)} \downarrow 0$, then it follows $T_2(x^{(\alpha)})_n=x^{(\alpha)}_{n-1}\downarrow_\alpha 0$ and thus $T_2(x^{(\alpha)})\downarrow 0$ in $Y$. That is, $T_2$ is order continuous. We conclude that $T=T_1-T_2\in L_{oc}(X,Y)$.

(d) We establish that $T\not\in L_{oc}(X,Z)$. 
For every fixed $n\in\NN$ define $e^{(n)}=(e^{(n)}_k)_{k\in\NN}\in X$ by $e^{(n)}_n=1$ and $e^{(n)}_k=0$ for $k\neq n$.
Similarly, for every fixed $n\in\NN$ let $y^{(n)}=(y^{(n)}_k)_{k\in\NN}\in X$ with $y^{(n)}_k=1$ for $k\geq n$ and zero otherwise. Then by (b)(i) for the sequence $y=(y^{(n)})_{n\in\NN}$ we have $y^{(n)} \downarrow 0$. Thus it follows $0\leq e^{(n)} \leq y^{(n)} \downarrow 0$, i.e.\ $e^{(n)}\xrightarrow{o} 0$ in $X$.

We show that $(T(e^{(n)}))_{n\in\NN}$ does not order converge in $Z$.
Observe that for every $n\in\NN$ we have $T(e^{(n)}) = x^{(n)}$ with $x^{(n)}$ as defined in (b)(iii).
Assume that $T$ is order continuous, then for the sequence $(x^{(n)})_{n\in\NN}$ in $Z$ we have $x^{(n)}\xrightarrow{o} 0$. In other words, there exists a sequence $(z^{(n)})_{n\in\NN}$ in $Z$ such that $\pm x^{(n)} \leq z^{(n)}\downarrow 0$. 
Fix $n\in\NN$ and let $m\in\NN$ be such that $m\geq n$. Then we have
\[ 1=x^{(m)}_m \leq z^{(m)}_m \leq z^{(n)}_m.\]
Due to $z^{(n)}\in Z$ we have
\[1\leq \lim_{m\rightarrow\infty} z^{(n)}_m = \sum_{k=1}^\infty \frac{z^{(n)}_{-k}}{2^k}.\]
Since this holds for every $n\in\NN$, we obtain that there exists an index $j\in\NN$ such that $z^{(n)}_{-j}\not\downarrow_n 0$ does not hold, otherwise by monotone convergence we would have $\sum_{k=1}^\infty \frac{z^{(n)}_{-k}}{2^k}\xrightarrow{n\rightarrow \infty} 0$. Thus there exists a real number $\delta\in\RR_{>0}$ such that for every $n\in\NN$ we have $z^{(n)}_{-j}\geq \delta$. Let $\omega\in Z$ be such that $\omega_{-j}=\delta$, $\omega_{-j-1}=-2\delta$ and $\omega_k=0$ for $k\in\ZZ$, where $k\neq -j$ and $k\neq -j-1$. It is immediate that $\omega\in Z$. Moreover, we have $\omega\leq z^{(n)}$ for every $n\in\NN$. However, $\omega\not\leq 0$, which is a contradiction to $\inf\left\{z^{(n)} \mid| n\in\NN\right\}=0$.
It follows that the sequence $(T(e^{(n)})){n\in\NN}$ does not order converge in $Z$ and therefore the operator $T\colon X\rightarrow Z$ is not order continuous.
\end{example}
\begin{numremark}\label{some_remark}
	(a) In Example~\ref{operators.7a} the two operators $T_1,T_2\colon X\rightarrow Y$ map into $Y$ but not into $Z$, even though we have $T_1-T_2=T\colon X\rightarrow Z$. The operator $T$ is not contained in  $L_{oc}(X,Z)$ and, hence, not in $L_{oc}^\diamond(X,Z)$. Therefore we have shown
	\[L(X,Z)\cap L_{oc}(X,Y)\cap L_r(X,Y)\neq L_{oc}^\diamond(X,Z).\]
	
	(b) 
    In the situation of  Example~\ref{operators.7a}    
    we have \[L_{oc}(X,Z)\neq L_{oc}^\diamond(X,Z).\]
    Indeed, we showed 	
	\[L_{oc}(X,Z^\delta)\cap L(X,Z)\not= L_{oc}(X,Z).\] 
	As $X$ is a vector lattice, by Proposition~\ref{atomic.0+1} we have $L_{oc}(X,Z^\delta)=L_{oc}^\diamond(X,Z^\delta)$.
		
	Now assume $L_{oc}(X,Z)= L_{oc}^\diamond(X,Z)$. Then by Corollary \ref{ocdiamondisnice} we get
	\[L_{oc}(X,Z)=L_{oc}^\diamond(X,Z)=L_{oc}^\diamond(X,Z^\delta)\cap L(X,Z)=L_{oc}(X,Z^\delta)\cap L(X,Z),\] 
	a contradiction.
\end{numremark}


\medskip
The following result is similar to Theorem~\ref{operators.7}, with the difference that here we extend not only the range but also the domain. The result yields order continuity for positive operators if they have a positive extension which is order continuous.
\begin{proposition}\label{operators.7x}
Let $X_1$ and $X_2$ be pre-Riesz spaces and $(Y_1,i_1)$, $(Y_2,i_2)$ their vector lattice covers, respectively. Let a positive operator $T\colon X_1\rightarrow X_2$ have a positive linear extension $\tilde{T}\colon Y_1\rightarrow Y_2$, i.e.\ $\tilde{T}\circ i_1=i_2\circ T$. Then the relation $\tilde{T}\in L_{oc}(Y_1,Y_2)$ implies $T\in L_{oc}(X_1,X_2)$.
\end{proposition}
\begin{proof}
 Let $(x_\alpha)_\alpha$ be a net in $X_1$ with $x_\alpha\downarrow 0$.
As $T$ is positive, it suffices to show $T(x_\alpha)\downarrow 0$.
By Lemma~\ref{properties.23y}\ref{properties.23y.eq1} the infimum of the set $\left\{i_1(x_\alpha)\mid| \alpha \in A\right\}$ exists in $Y_1$ and equals $0$, i.e.\ $i_1(x_\alpha)\downarrow 0$.
As $\tilde{T}$ is order continuous and  positive, it follows that $\tilde{T}(i_1(x_\alpha))\downarrow 0$ in $Y_2$. 
We use 
$\tilde{T}(i_1(x_\alpha))=i_2(T(x_\alpha))$ in $Y_2$. By Lemma~\ref{properties.23y}\ref{properties.23y.eq2} the relation $i_2(T(x_\alpha))\downarrow 0$ in $Y_2$ implies $T(x_\alpha)\downarrow 0$ in $X_2$.
\end{proof}

\section{Embedding a pre-Riesz space of order continuous operators: a counterexample}\label{atomic.sec1}

Let $X$ and $Z$ be ordered vector spaces, where $X$ is directed and $Z$ is Archimedean. Then $L_r(X,Z)$ is directed and Archimedean and therefore a pre-Riesz space. By Proposition~\ref{atomic.0+2}, $L_{oc}^\diamond(X,Z)$ is an Archimedean pre-Riesz space as well. In general, $L_r(X,Z)$ and $L_{oc}^\diamond(X,Z)$ are not vector lattices, see below. A natural next step is to look for suitable vector lattice covers of $L_r(X,Z)$ and $L_{oc}^\diamond(X,Z)$, respectively, which can be represented as spaces of operators.

If, in addition, $X$ has the RDP, then by Theorem~\ref{prelim.53} the operator space  $L_r(X,Z^\delta)$ is a Dedekind complete vector lattice. Furthermore, by Theorem~\ref{prelim.1004} the space $L_{oc}(X,Z^\delta)$ is a band in $L_r(X,Z^\delta)$ and thus by Lemma~\ref{properties.0} likewise a Dedekind complete vector lattice. An idea for a vector lattice cover of $L_r(X,Z)$ might be to look at the canonical embedding $L_r(X,Z)\sub L_r(X,Z^\delta)$.
Moreover, by Corollary~\ref{operators.xxx} we have $L_{oc}^\diamond(X,Z)\sub L_{oc}(X,Z^\delta)$. The question arises whether these inclusions are order dense.

The first inclusion is dealt with in \c[Counterexamples~2.10 and 2.13]{Schouten}. There, for the normed vector lattice $X:=\ell_p$ of $p$-summable sequences, where $p\in \hs]1,\infty[$, and for $Z=C[0,1]$ it is shown that $L_r(X,Z)$ is not majorizing in $L_r(X,Z^\delta)$ and thus not order dense.

We deal with the second inclusion $L_{oc}^\diamond(X,Z)\sub L_{oc}(X,Z^\delta)$ and give an example where this inclusion is not order dense.
To that end, we consider the vector lattice $X=Z=\ell_0^\infty$ of eventually constant sequences\footnote{We focus on this example, since the order continuous operators are characterized in \c{3}.}. As is established in \c[Theorem~5.1]{AbrWick1991} by Abramovich and Wickstead, the pre-Riesz space $L_r(\ell_0^\infty)$ is not a vector lattice. A vector lattice cover of $L_r(\ell_0^\infty)$ is constructed in \c{3}. However, in this example the vector lattice cover is not given as a space of operators. In Example~\ref{atomic.3a} below we establish that $L_{oc}^\diamond(\ell_0^\infty)$ is not a vector lattice.
Notice that $Z^\delta$ is the space $\ell^\infty$ of bounded sequences. Example~\ref{atomic.1} demonstrates that $L_r(\ell_0^\infty)$ is not order dense in $L_r(\ell_0^\infty,\ell^\infty)$. Based on this result, in Example~\ref{atomic.1x} we show that $L_{oc}^\diamond(\ell_0^\infty)$ is not majorizing and therefore not order dense in $L_{oc}(\ell_0^\infty,\ell^\infty)$. Hence, $L_{oc}(\ell_0^\infty,\ell^\infty)$ is not a vector lattice cover of $L_{oc}^\diamond(\ell_0^\infty)$. 
This shows that the straightforward idea to make the range space Dedekind complete is not a suitable approach to obtain a vector lattice cover.

\begin{example}\label{atomic.1}
\textit{The space $L_r(\ell_0^\infty)$ is not majorizing and therefore not order dense in $L_r(\ell_0^\infty,\ell^\infty)$.}

Let $\ell_0^\infty$ be the vector lattice of all eventually constant sequences on $\NN$ with the standard order. In \c[Theorem~5.1]{AbrWick1991} it is shown that the operator space $L_r(\ell_0^\infty)$ does not have the RDP and therefore is not a vector lattice. However, as $\ell_0^\infty$ is Archimedean and the space of regular operators is automatically directed, it follows that $L_r(\ell_0^\infty)$ is pre-Riesz.

The vector lattice $\ell_0^\infty$ is not Dedekind complete, and $\ell^\infty$ is its Dedekind completion. Denote by $e_i$ the sequences with $1$ at the $i$th entry and $0$ elsewhere, and by $\1$ the constant-one sequence.
The set $\mathbfcal{B}:=\left\{e_i\mid|i\in\NN\right\}\cup\left\{\1\right\}$ yields an algebraic basis $(\1, e_1,e_2,\ldots)$ of $\ell_0^\infty$.
In the subsequent table we define a mapping $T\in L_r(\ell_0^\infty,\ell^\infty)$ on the basis elements and extend $T$ linearly. The blanks are filled by $0$.

\begin{tabular}{c | l}
$b\in \mathbfcal{B}$		&	$T(b)$	\\
\hline
$e_1$ &	(\textbf{1} \quad 0 \quad \textbf{1} \quad 0 \quad 0 \quad \textbf{1} \quad 0 \quad 0 \quad 0 \quad \textbf{1} \quad 0 \quad 0 \quad 0 \quad 0 \quad $\hdots$)\\
$e_2$ &	(0 \quad \textbf{1} \quad 0 \quad \textbf{1} \quad 0 \quad 0 \quad \textbf{1} \quad 0 \quad 0 \quad 0 \quad \textbf{1} \quad 0 \quad 0 \quad 0 \quad $\hdots$)\\
$e_3$ &	(\phantom{0} \quad \phantom{0} \quad 0 \quad 0 \quad \textbf{1} \quad 0 \quad 0 \quad \textbf{1} \quad 0 \quad 0 \quad 0 \quad \textbf{1} \quad 0 \quad 0 \quad $\hdots$)\\
$e_4$ &	(\phantom{0} \quad \phantom{0} \quad \phantom{0} \quad \phantom{0} \quad \phantom{0} \quad 0 \quad 0 \quad 0 \quad \textbf{1} \quad 0 \quad 0 \quad 0 \quad \textbf{1} \quad 0 \quad $\hdots$)\\
$e_5$ &	(\phantom{0} \quad \phantom{0} \quad \phantom{0} \quad \phantom{0} \quad \phantom{0}  \quad \phantom{0} \quad \phantom{0} \quad \phantom{0} \quad \phantom{0} \quad 0 \quad 0 \quad 0 \quad 0 \quad \textbf{1} \quad $\hdots$)\\
$\vdots$ & $\hdots$\\
$\1$ & (1 \quad 1 \quad 1 \quad 1 \quad 1 \quad 1 \quad 1 \quad 1 \quad 1 \quad 1 \quad 1 \quad 1 \quad 1 \quad 1 \quad $\hdots$).
\end{tabular}

It is clear that the operator $T$ maps sequences in $\ell_0^\infty$ to sequences in $\ell^\infty$. We establish that $T$ is positive (and, hence, in $L_r(\ell_0^\infty,\ell^\infty)$) and can not be majorized by an operator in $L_r(\ell_0^\infty)$.

To show that $T$ is positive, notice first that for every $n\in\NN$ we have $\sum_{i=1}^{n}T(e_i)\leq T(\1)$. Let $x=\lambda_0 \1 + \sum_{i=1}^n \lambda_i e_i \in \ell_0^\infty$ be a positive element, i.e. for every $i\in\NN$ we have $\lambda_0 + \lambda_i \geq 0$. Thus
\begin{align*}
T(x) \hs &=\hs \lambda_0 T(\1) + \sum_{i=1}^n \lambda_i T(e_i) \hs\geq\hs \lambda_0 \sum_{i=1}^{n}T(e_i) + \sum_{i=1}^n \lambda _i T(e_i) \hs=\hs\\
 &=\hs \sum_{i=1}^{n}(\lambda_0+\lambda_i)T(e_i) \hs\geq\hs 0,
\end{align*}
i.e.\ $T$ is positive and therefore regular. Thus $T\in L_r(\ell_0^\infty,\ell^\infty)$.

It is left to show that $T$ can not be majorized by an operator in $L_r(\ell_0^\infty)$.
Assume there exists an operator $S\in L_r(\ell_0^\infty)$ with $T\leq S$ (note that we view $S\in L_r(\ell_0^\infty)$ as an operator in $L_r(\ell_0^\infty,\ell^\infty)$). Due to $T\geq 0$ it follows $S\geq 0$. Furthermore, for each element $b\in \mathbfcal{B}$ we get $T(b) \leq S(b)$ and thus $\limsup T(b) \leq \limsup S(b)$. Since $S(b)$ is an eventually constant sequence, we have $\limsup S(b) = \lim S(b)$. In particular, we have
\begin{equation}\label{atomic.1.eq1}
\forall i\in\NN\colon \quad 1 = \limsup T(e_i) \leq \lim S(e_i).
\end{equation}
Moreover, for each $n\in\NN$ the inequality $\sum_{i=1}^n e_i \leq \1$ is valid.
Since $S$ is positive, it follows for every $n\in\NN$
\[\sum_{i=1}^n S(e_i) \leq S(\1) \quad \t{ and }\quad  \sum_{i=1}^n \lim S(e_i) = \lim \left(\sum_{i=1}^n S(e_i)\right) \leq \lim S(\1).\]
However, due to \eqref{atomic.1.eq1} we have for every $n\in\NN$ \[n \leq \sum_{i=1}^n \lim S(e_i) \leq \lim S(\1),\]
which is a contradiction to $S(\1)$ being an eventually constant sequence.
\end{example}
\begin{remark}
The space $X:=\ell_0^\infty$ is a normed vector lattice, has an algebraic basis and $\1$ as an order unit. The above example shows that even under these strong additional conditions $L_r(X,X)$ is not order dense in $L_r(X,X^\delta)$.

Note that in the above mentioned Example in \c{Schouten} the underlying spaces are Banach lattices. Moreover, this example is more elaborate and uses the norms in the computations.
\end{remark}

\begin{example}\label{atomic.3a}
\textit{The pre-Riesz space $L_{oc}^\diamond(\ell_0^\infty)$ is pervasive, but not a vector lattice.}

In the subsequent steps (i) to (vi) we recall the setting and the results of \c[Example~3.1]{3}.

(i)  The authors consider a lattice isomorphism between $\ell_0^\infty$ and the space $c_{00}$ of finite sequences on $\NN_0$ endowed with the order induced by a certain cone $K$. Using this isomorphism, in \c[Example~3.1~(a),~(b)~and~(j)]{3} they represent positive, regular and order continuous operators on $\ell_0^\infty$ as matrices. We recall the procedure. Notice that the cone $K$ on $c_{00}$ is given by
\[K = \left\{ (x_i)_{i\in\NN_0}\in c_{00} \mid| x_0\geq 0 \t{ and } x_0+x_i \geq 0 \t{ for every } i\in\NN\right\}.\]
Let $\mathbfcal{B}$ be as in Example~\ref{atomic.1}. The mapping
\begin{eqnarray*}
f\colon \left\{\1,e_1,e_2,\ldots\right\} \sub \ell_0^\infty &\longrightarrow& c_{00}\\
\1 &\longmapsto& e_0,\quad e_i \longmapsto e_i \quad (i\in\NN)
\end{eqnarray*}
is a basis isomorphism and can be extended linearly to all of $\ell_0^\infty$. This extension which we also call $f$ is an order isomorphism between the vector lattices $\ell_0^\infty$ (with pointwise order)  and $(c_{00},K)$.
Thus every linear operator $\hat{A}$ on $\ell_0^\infty$ induces a linear operator $\tilde{A}$ on $c_{00}$ by the relationship $\tilde{A}:= f \circ \hat{A} \circ f^{-1}$, following the scheme 
\[
\begin{xy}
  \xymatrix{
      (\ell_0^\infty, (\ell_0^\infty)_+) \ar[rr]^{\hat{A}} \ar[d]_f    & &  (\ell_0^\infty, (\ell_0^\infty)_+) \ar[d]^f \\ 
      (c_{00},K) \ar[rr]^{\tilde{A}}             & &  (c_{00},K).
  }
\end{xy}
\]
Furthermore, one represents $\hat{A}$ as an infinite matrix $A:=(a_{ij})_{i,j\in\NN_0}$ with columns which are eventually zero, where $a_{i0}$ is the $i$th component of $f(\hat{A}\1)$ and $a_{ij}$ is the $i$th component of $f(\hat{A}e_j)$, for $j\in\NN$ and $i\in\NN_0$. We use the notation
\[\mathcal{L} := \left\{ A\in \RR^{\NN_0 \times \NN_0} \mid| \t{the colums of } A \t{ are eventually zero}\right\}.\]
We will say that a matrix in $\mathcal{L}$ is positive, regular or order continuous if it represents a positive, regular or order continuous operator on $\ell_0^\infty$, respectively. Let $A =(a_{ij})_{i,j\in\NN_{0}}\in \mathcal{L}$.

(ii) The matrix $A$ is \textit{positive} if and only if for every fixed $i\in\NN$ the following two conditions are satisfied:
\begin{enumerate}
\item[($\alpha$)]\label{atomic.3a.it2.italpha} $a_{0j}+a_{ij}\geq 0$ for all $j\in\NN$.
\item[($\beta$)]\label{atomic.3a.it2.itbeta} $a_{00}+a_{i0} \geq \sum_{j=1}^\infty (a_{0j}+a_{ij})$.
\end{enumerate}
(iii) The matrix $A$ is \textit{regular} if and only if the sequence $\left(\sum_{j=0}^\infty|a_{ij}|\right)_{i\in\NN_0}$ of absolute row sums is bounded. The set of all regular matrices is denoted by
\[\mathcal{R}:=\left\{A=(a_{ij})_{i,j}\in\RR^{\NN_0 \times \NN_0} \mid|  A \in\mathcal{L} \t{ and }\left(\sum_{j=0}^\infty|a_{ij}|\right)_{i\in\NN_0} \t{ is bounded}\right\}.\]
(iv) The matrix $A$ is \textit{order continuous} if and only if for every fixed $i\in\NN$ we have
\[\sum_{j=1}^\infty (a_{ij}+a_{0j}) = a_{00} + a_{i0}.\]
The set of all order continuous matrices is denoted by
\[\mathcal{N}:=\left\{A=(a_{ij})_{i,j}\in\RR^{\NN_0 \times \NN_0} \mid|  A \in\mathcal{R} \t{ and }\forall i\in\NN\colon \sum_{j=1}^\infty (a_{ij}+a_{0j}) = a_{00} + a_{i0}\right\}.\]
(v) By \c[Theorem~4.1]{AbrWick1991} we have $L_r(\ell_0^\infty) = L_b(\ell_0^\infty)$ and by \c[Theorem~5.1]{AbrWick1991} the pre-Riesz space $L_r(\ell_0^\infty)$ which corresponds to $\mathcal{R}$ does not have the RDP and thus is not a vector lattice. In \c[Example~3.1~(c),~(d)~and~(e)]{3} a vector lattice cover $Y$ of $\mathcal{R}$ is computed. The vector lattice $Y$ consists of certain matrices, such that the order on $Y$ is given componentwise. We briefly recall the construction.

Consider the following properties of an infinite matrix $B=(b_{ij})_{ij}\in\RR^{\NN \times \NN_0}$:
\begin{enumerate}
\item[(a)]\label{atomic.3a.it5.ita} For every $j\in\NN$ the column sequence $(b_{ij})_{i\in\NN}$ is eventually constant.
\item[(b)]\label{atomic.3a.it5.itb} For column limits $\disp\beta_j:=\lim_{i\to\infty} b_{ij}$, where $j\in\NN$, we have $\disp\sum_{j=1}^\infty |\beta_j| < \infty$.
\item[(c)]\label{atomic.3a.it5.itc} The column sequence $(b_{i0})_{i\in\NN}$ is bounded.
\item[(d)]\label{atomic.3a.it5.itd} The sequence of absolute row sums $\disp\left( \sum_{j=1}^\infty |b_{ij}| \right)_{\hspace*{-1.5mm}i\in\NN}$  is bounded.
\end{enumerate}
The space
$Y := \left\{ B\in\RR^{\NN \times \NN_0} \mid| B \t{ satisfies (a), (b), (c) and (d)} \right\}$
with coordinatewise order is a vector lattice, and the mapping
\begin{align*}
F\colon &\mathcal{R} \longrightarrow Y \\
(a_{ij}&)_{ij} \mapsto \begin{cases}
   					a_{0j} + a_{ij} & \t{ for } i,j\in\NN,\\
   					a_{00} +a_{i0} -\sum_{k=1}^\infty (a_{0k} + a_{ik})
   							& \t{ for } j=0\t{ and } i\in\NN.
   						\end{cases}
\end{align*}
is an order dense embedding of $\mathcal{R}$ into $Y$, i.e.\ $Y$ is a vector lattice cover of $\mathcal{R}$, as was shown in \c[Example~3.1~(e)]{3}. Moreover,  $\mathcal{R}$ is even pervasive. Notice that it is not clear whether the elements of the space $Y$ can be viewed as operators, or rather, if one can find a suitable space on which the elements of $Y$ are acting as operators.

The acting of $F$ on a matrix $A= (a_{ij})_{ij}$ can be interpreted as three separate successive operations:
\begin{enumerate}
\item[(I)]\label{atomic.3a.it5.itI} Add the $0$th row of $A$ onto each other row.
\item[(II)]\label{atomic.3a.it5.itII} Delete the $0$th row of the resulting matrix.
\item[(III)]\label{atomic.3a.it5.itIII} In the resulting matrix fix a row $i\in\NN$. Sum all entries of the row sequence starting from $j=1$ and subtract this sum from the $0$th entry of the sequence. Replace the $0$th entry of the sequence by this result. Let the other entries of the row unchanged.
Do this procedure with each fixed row $i$ of $A$ for every $i\in\NN$.
\end{enumerate}

(vi) It is shown in \c[Example~3.1~(j)]{3} that $\mathcal{N}$ is a directed subspace of the Archimedean pre-Riesz space $\mathcal{R}$ and therefore is pre-Riesz. Moreover, \c[Example~3.1~(k)]{3} yields that $\mathcal{N}$ is a band in $\mathcal{R}$ and we have
\[\mathcal{N} = \left\{A \in\mathcal{R} \mid| \forall i\in\NN\colon F(A)_{i0}=0\right\}.\]
An extension band of $\mathcal{N}$ in $Y$ is given by
\[\mathcal{B}:=\left\{B=(b_{ij})_{ij}\in\RR^{\NN \times \NN_0} \mid| B\in Y \t{ and }\forall i\in\NN\colon b_{i0}=0\right\}.\]

We use the statements (i) to (vi) to show that the space $L_{oc}^\diamond(\ell_0^\infty)$ (which corresponds to $\mathcal{N}$) is not a vector lattice. First notice that by (vi) the space
$L_{oc}(\ell_0^\infty)$
is directed and thus $L_{oc}^\diamond(\ell_0^\infty) = L_{oc}(\ell_0^\infty)$. We provide a vector lattice cover of the pre-Riesz space $\mathcal{N}$, making use of the vector lattice cover $Y$ of $\mathcal{R}$ and of the extension band $\mathcal{B}$ of $\mathcal{N}$. Moreover, we will use Theorem~\ref{closedness.1001}. As $\mathcal{B}$ is a subset of $Y$, every matrix in $\mathcal{B}$ admits the conditions (a)-(d) from (v). Therefore we can describe every matrix $A=(a_{ij})_{i\in\NN,j\in\NN_o}\in\mathcal{B}$ using the properties (a), (b), (d) and
\begin{center}
(C) The $0$th column $(a_{i0})_i$ is zero,
\end{center}
i.e.\ 
$\mathcal{B} =\left\{B\in\RR^{\NN \times \NN_0} \mid| \t{(a), (b), (C) and (d) hold.}\right\}$.

By (vi) the matrix space $\mathcal{N}$ is a directed band in $\mathcal{R}$ and, in particular, a directed ideal. 
Theorem~\ref{something}~\ref{something.i4} yields that $\mathcal{N}$ has an extension ideal in $Y$.
In particular, the smallest extension ideal $Y_{oc}$ of $\mathcal{N}$ exists in $Y$. By Theorem~\ref{something}~\ref{something.i2} and \ref{something.i4} it follows that $\mathcal{N}$ is majorizing in $Y_{oc}$, hence \[Y_{oc}=\left\{B\in Y\mid| \exists A\in \mathcal{N}\colon |B|\leq A\right\}.\]
As $Y_{oc}$ is an ideal in a vector lattice $Y$, it follows that $Y_{oc}$ is a vector lattice in its own right, endowed with the coordinatewise order induced by $Y$.

By (v) we have that $\mathcal{R}$ is pervasive in $Y$. Moreover, since $\mathcal{N}$ is a directed band in $\mathcal{R}$, it is, in particular, a directed ideal. Therefore by Theorem~\ref{closedness.1001} we obtain that $Y_{oc}$ is a vector lattice cover of the pre-Riesz space $\mathcal{N}$ and that $\mathcal{N}$ is pervasive. 

\medskip
Using the fact that $Y_{oc}$ is a vector lattice cover of the pre-Riesz space $\mathcal{N}$ we now can proceed to establish that $\mathcal{N}$ is not a vector lattice.
Consider two operators
$\hat{P},\hat{Q}\in L_r(\ell_0^\infty)$ which are represented by the matrices
\begin{equation*}
P:=\left( \begin{array}{ccccccc}
2 & 1 & 0 & 0 & 0 & 0 & \hdots \\
0 & 1 & 0 & 0 & 0 & 0 & \hdots \\
0 & 0 & 1 & 0 & 0 & 0 & \hdots \\
0 & 0 & 0 & 1 & 0 & 0 & \hdots \\
0 & 0 & 0 & 0 & 1 & 0 & \hdots \\
0 & 0 & 0 & 0 & 0 & 1 & \hdots \\
\vdots & \vdots & \vdots & \vdots & \vdots & \vdots & \ddots
\end{array} \right)
\quad\t{ and }\quad
Q:=\left( \begin{array}{ccccccc}
2 & 1 & 0 & 0 & 0 & 0 & \hdots \\
0 & 1 & 0 & 0 & 0 & 0 & \hdots \\
0 & 0 & 0 & 1 & 0 & 0 & \hdots \\
0 & 0 & 0 & 1 & 0 & 0 & \hdots \\
0 & 0 & 0 & 0 & 0 & 1 & \hdots \\
0 & 0 & 0 & 0 & 0 & 1 & \hdots \\
\vdots & \vdots & \vdots & \vdots & \vdots & \vdots & \ddots
\end{array} \right)
\end{equation*}
in $\mathcal{R}$. Calculating the images of $P$ and $Q$ under the embedding map $F$ in (v), we obtain
\begin{equation*}
F(P)=\left( \begin{array}{ccccccc}
0 & 2 & 0 & 0 & 0 & 0 & \hdots \\
0 & 1 & 1 & 0 & 0 & 0 & \hdots \\
0 & 1 & 0 & 1 & 0 & 0 & \hdots \\
0 & 1 & 0 & 0 & 1 & 0 & \hdots \\
0 & 1 & 0 & 0 & 0 & 1 & \hdots \\
\vdots & \vdots & \vdots & \vdots & \vdots & \vdots & \ddots
\end{array} \right)
\quad\t{ and }\quad
F(Q)=\left( \begin{array}{ccccccc}
0 & 2 & 0 & 0 & 0 & 0 & \hdots \\
0 & 1 & 0 & 1 & 0 & 0 & \hdots \\
0 & 1 & 0 & 1 & 0 & 0 & \hdots \\
0 & 1 & 0 & 0 & 0 & 1 & \hdots \\
0 & 1 & 0 & 0 & 0 & 1 & \hdots \\
\vdots & \vdots & \vdots & \vdots & \vdots & \vdots & \ddots
\end{array} \right),
\end{equation*}
i.e.\ $P,Q\in\mathcal{N}$ by (vi).
Therefore $\hat{P},\hat{Q}\in L_{oc}^\diamond(\ell_0^\infty)$. Moreover, we have $F(P), F(Q) \in Y_{oc}\sub\mathcal{B}$. We can easily calculate the infimum of $F(P)$ and $F(Q)$ in $Y_{oc}$, as the lattice operations in $Y_{oc}$ are defined pointwise. It follows
\begin{equation*}
M_2:=F(P)\Inf F(Q)=\left( \begin{array}{ccccccc}
0 & 2 & 0 & 0 & 0 & 0 & \hdots \\
0 & 1 & 0 & 0 & 0 & 0 & \hdots \\
0 & 1 & 0 & 1 & 0 & 0 & \hdots \\
0 & 1 & 0 & 0 & 0 & 0 & \hdots \\
0 & 1 & 0 & 0 & 0 & 1 & \hdots \\
\vdots & \vdots & \vdots & \vdots & \vdots & \vdots & \ddots
\end{array} \right).
\end{equation*}
Assume that $\mathcal{N}$ is a vector lattice. Then the infimum $M:=P\Inf Q$ exists in $\mathcal{N}$. By Lemma~\ref{properties.23y}~\ref{properties.23y.eq1} it follows $F(M)=F(P\Inf Q)=F(P)\Inf F(Q)=M_2$ in $Y_{oc}$. We apply the steps (III) and (II) of (v) backwards to the matrix $M_2=(m^{(2)}_{ij})_{i\in\NN,j\in\NN_o}$ to obtain a matrix $M_1=(m^{(1)}_{ij})_{ij\in\NN_0}$. Then we show that there is no matrix $M\in\mathcal{N}$ to which we can apply step (I) of (v) to obtain $M_1$. That is, the matrix $M_2=F(P)\Inf F(Q)$ has no pre-image under the embedding $F$ and therefore $\mathcal{N}$ is not a vector lattice.

Let us apply the reversed step (III) of (v) to the matrix $M_2$, i.e.\ for every fixed row $i$ we add the row sum $\sum_{j=1}^\infty m^{(2)}_{ij}$ (starting with $j=1$) to the first entry $m^{(2)}_{i0}$ of the row $i$. Moreover, we apply the reversed step (II) of (v) by prepending a new $0$th row with not yet determined entries. Then we obtain
\begin{equation*}
M_1=\left( \begin{array}{ccccccc}
m^{(1)}_{00} & m^{(1)}_{01} & m^{(1)}_{02} & m^{(1)}_{03} & m^{(1)}_{04} & m^{(1)}_{05} & \hdots \\
2 & 2 & 0 & 0 & 0 & 0 & \hdots \\
1 & 1 & 0 & 0 & 0 & 0 & \hdots \\
2 & 1 & 0 & 1 & 0 & 0 & \hdots \\
1 & 1 & 0 & 0 & 0 & 0 & \hdots \\
2 & 1 & 0 & 0 & 0 & 1 & \hdots \\
\vdots & \vdots & \vdots & \vdots & \vdots & \vdots & \ddots
\end{array} \right).
\end{equation*}
Notice that the $0$th  column of $M_1$ is alternating. We apply the first step (I) of (v) to a matrix $M=(m_{ij})_{ij\in\NN_0}\in\mathcal{N}$. That is, we add the $0$th row of $M$ onto each other row of $M$ to obtain $M_1$. With this for the entries of the $0$th row of $M$ we obtain equations $m_{00}+m_{i0}=2$ for even $i\in\NN$ and $m_{00}+m_{i0}=1$ for odd $i\in\NN$. As $M\in\mathcal{R}$, the columns of $M$ have only finitely many non-zero entries. In particular, the $0$th column of $M$ is eventually zero. Therefore we can find an even $N\in\NN$ such that $m_{N0}=m_{(N+1)0}=0$. Then it follows $2=m_{00}+m_{N0}=m_{00}$ and $1=m_{00}+m_{(N+1)0}=m_{00}$, a contradiction. Thus the matrix $M_2=F(P)\Inf F(Q)$ has no pre-image under the embedding $F$. Therefore $\mathcal{N}$ and, consequenctly, $L_{oc}^\diamond(\ell_0^\infty)$, is not a vector lattice.
\end{example}

In the following example we return to the operator $T$ from Example~\ref{atomic.1} and show that $T$ is order continuous.

\begin{example}\label{atomic.1x}
\textit{The pre-Riesz space $L_{oc}^\diamond(\ell_0^\infty)$ is not majorizing and therefore not order dense in $L_{oc}(\ell_0^\infty,\ell^\infty)$.}

Let the operator $T\colon \ell_0^\infty\rightarrow \ell_\infty$ be as in Example~\ref{atomic.1}. We have seen that $T$ is positive and therefore regular, i.e.\ $T\in L_r(\ell_0^\infty,\ell^\infty)$.
We now establish that $T$ is order continuous. Since $T$ is positive, it is sufficient to show that for every net $(x_\alpha)_\alpha$ in $\ell_0^\infty$ the condition $x_\alpha\downarrow 0$ implies $T(x_\alpha)\downarrow 0$ in $\ell^\infty$. Let $(x_\alpha)_\alpha$ be a net in $\ell_0^\infty$ with $x_\alpha\downarrow 0$. Every $x_\alpha$ can be represented as a finite sum of elements of the algebraic basis $\mathbfcal{B}$, i.e.\ there exists an $N_\alpha\in \NN$ and real numbers $\lambda_\alpha,\lambda_\alpha^{(k)}$ for $0\leq k\leq N_\alpha$ such that
\begin{equation*}\label{yay2}
x_\alpha \hs = \hs \lambda_\alpha \1 + \sum_{k=1}^{N_\alpha} \lambda_\alpha^{(k)} e_k.
\end{equation*}
For $k\in\NN$ with $k>N_\alpha$ let $\lambda_\alpha^{(k)}:=0$.
Consider the $k$th component of the sequence $x_\alpha\in\ell_0^\infty$, i.e.\
$x_\alpha^{(k)} = \lambda_\alpha + \lambda_\alpha^{(k)}$.
Since the order on $\ell_0^\infty$ is pointwise and $x_\alpha\downarrow 0$, for the component net $(x_\alpha^{(k)})_\alpha$ of reals we have $x_\alpha^{(k)}\downarrow 0$.

By the definition of $T$ in Example~\ref{atomic.1} for every fixed $\alpha$ we obtain
\[T(x_\alpha) = (\lambda_\alpha+\lambda_\alpha^{(1)},\hs\lambda_\alpha+\lambda_\alpha^{(2)},\hs\lambda_\alpha+\lambda_\alpha^{(1)},\hs\lambda_\alpha+\lambda_\alpha^{(2)},\hs\lambda_\alpha+\lambda_\alpha^{(3)},\hs\ldots).\]
Hence for every $n\in\NN$ there exists a $k\in\NN$ such that
\[ (T(x_\alpha))^{(n)} \hs=\hs \lambda_\alpha + \lambda_\alpha^{(k)}.\]
Notice that $k$ does not depend on $\alpha$.
Thus for a fixed $n\in\NN$ we obtain a net $((T(x_\alpha))^{(n)})_\alpha$ in $\RR$ with
\[(T(x_\alpha))^{(n)} =\lambda_\alpha + \lambda_\alpha^{(k)}= x_\alpha^{(k)}  \downarrow_\alpha 0.\]
For the net $(T(x_\alpha))_\alpha$ in $\ell^\infty$ it follows $T(x_\alpha)\downarrow 0$. We conclude that the operator $T\colon \ell_0^\infty\rightarrow\ell^\infty$ is order continuous, i.e.\ $T\in L_{oc}(\ell_0^\infty,\ell^\infty)$.

By Example~\ref{atomic.1} the operator $T$ 
can not be majorized by an operator in $L_r(\ell_0^\infty)$ and, consequently, also not by an operator in $L_{oc}^\diamond(\ell_0^\infty)$. Thus $L_{oc}^\diamond(\ell_0^\infty)$ is not order dense in $L_{oc}(\ell_0^\infty,\ell^\infty)$.
\end{example}

The above example shows that even under strong additional conditions on $X$ (such as being a vector lattice with an algebraic basis and an order unit) the space $L_{oc}^\diamond(X)$ is not majorizing in $L_{oc}(X,X^\delta)$, in general.

To obtain a subspace of $L_{oc}(X,X^\delta)$ in which $L_{oc}^\diamond(X)$ is majorizing, one can consider the ideal\footnote{Since $L_{oc}^\diamond(X)$ is a directed subspace of $L_{oc}(X,X^\delta)$, the space $L_{oc}^\diamond(X)$ is majorizing in $J$.} $J$ generated by $L_{oc}^\diamond(X)$ in $L_{oc}(X,X^\delta)$.
We conjecture that for $X=\ell_0^\infty$ the ideal $J$ is a vector lattice cover of $L_{oc}^\diamond(X)$. The question remains open for which directed ordered vector spaces $X$ having the RDP and Archimedean ordered vector spaces $Z$ the (Dedekind complete) ideal $J$ generated by $L_{oc}^\diamond(X,Z)$ in $L_{oc}(X,Z^\delta)$ is a vector lattice cover of $L_{oc}^\diamond(X,Z)$. In this case, by Lemma~\ref{properties.0} the ideal $J$ would be even the Dedekind completion of $L_{oc}^\diamond(X,Z)$.


\bibliographystyle{plain}

\end{document}